\documentclass[reqno, 10pt, a4paper]{amsart}

\usepackage{fullpage}
\selectfont

\usepackage{algorithm,algpseudocode}

\algnewcommand{\IIf}[1]{\State\algorithmicif\ #1\ \algorithmicthen}
\algnewcommand{\EndIIf}{\unskip\ \algorithmicend\ \algorithmicif}

\usepackage[utf8]{inputenc}
\usepackage[OT2,T1]{fontenc}
\usepackage[english]{babel}

\usepackage{amsmath}
\usepackage{amsfonts}
\usepackage{amssymb}
\usepackage{amsthm}

\usepackage{mathrsfs}
\usepackage{listings}
\usepackage[mathcal]{euscript}
\usepackage{bbm}%

\usepackage{enumitem}
\usepackage{multirow}

\usepackage[dvipsnames]{xcolor}
\usepackage{graphicx}

\usepackage{tikz-cd}
\usetikzlibrary{matrix,positioning}
\usepackage[all]{xy}

\usepackage{colortbl}
\definecolor{mygray}{gray}{0.92}

\usepackage{array}
\newcolumntype{C}[1]{>{\centering\arraybackslash$}p{#1}<{$}}

\usepackage{breqn}
\newcounter{myequation}[equation]

\usepackage{float}
\restylefloat{table}
\usepackage{multirow}

\usepackage{hyperref}

\theoremstyle{plain}
\newtheorem{theorem}{Theorem}[section]

\newtheorem{proposition}[theorem]{Proposition}
\newtheorem{lemma}[theorem]{Lemma}

\theoremstyle{definition}
\newtheorem{definition}[theorem]{Definition}

\theoremstyle{remark}
\newtheorem{remark}[theorem]{Remark}

\numberwithin{equation}{section}

\def\epsilon{\varepsilon}

\def\tilde{\widetilde}

\DeclareMathOperator{\Aut}{Aut}

\DeclareMathOperator{\Cl}{Cl}

\DeclareMathOperator{\cont}{cont}

\DeclareMathOperator{\disc}{disc}

\DeclareMathOperator{\End}{End}
\DeclareMathOperator{\Ends}{End^{sym}}

\DeclareMathOperator{\Gal}{Gal}
\DeclareMathOperator{\GL}{GL}

\DeclareMathOperator{\Hom}{Hom}

\DeclareMathOperator{\id}{Id}

\DeclareMathOperator{\NS}{NS}

\DeclareMathOperator{\SL}{SL}

\DeclareMathOperator{\tr}{tr}

\DeclareMathOperator{\Tr}{Tr}

\DeclareMathOperator{\Ht}{\textrm{Herm}^{\textrm{twist}}_2}

\newcommand{\dq}{{/\kern -3pt/}}

\def\a{\mathfrak{a}}
\def\b{\mathfrak{b}}

\def\C{\mathbb{C}}

\def\O{\mathcal{O}}

\def\Q{\mathbb{Q}}

\def\Z{\mathbb{Z}}

\usepackage{bm}

\def\Ac{\mathcal{A}}

\def\Oc{\mathcal{O}}
\def\Pc{\mathcal{P}}

\def\Sc{\mathcal{S}}

\hypersetup{
  pdfauthor   = {Lercier, Liu, Lorenzo Garc\'ia, Ritzenthaler},
  pdftitle    = {Reduction type of smooth quartics},
  pdfsubject  = {},
  pdfkeywords = {},
  backref=true, pagebackref=true, hyperindex=true, colorlinks=true,
  breaklinks=true, urlcolor=blue, linkcolor=blue, citecolor=blue,
  bookmarks=true, bookmarksopen=true
}

\begin{document}

\title{Refined Humbert invariants and subvarieties of $\Ac_2(\C)$: the rank 3 case}
\date{\today}

\begin{abstract}
Over the complex field, we show that the principally polarised abelian surfaces with a given refined Humbert invariant of rank 3 and discriminant $-16d_K$, where $d_K<0$ is a fundamental discriminant, are all Galois conjugate.
    \end{abstract}


\author[Lorenzo]{Elisa Lorenzo Garc\'ia}
\address{%
	Elisa Lorenzo Garc\'ia,
   CNRS \& Aix-Marseille Université, Institut de Mathématiques de Marseille, Campus de
Luminy, 13009 Marseille, France. %
}
\email{elisa.lorenzo-garcia@univ-amu.fr}

\author[Ritzenthaler]{Christophe Ritzenthaler}
\address{%
	Christophe Ritzenthaler,
  Univ Rennes, CNRS, IRMAR - UMR 6625, F-35000
 Rennes, %
  France. %
  }
  
\address{%
  Université Côte d'Azur, CNRS, LJAD UMR 7351,
  Nice,
  France
}

\email{christophe.ritzenthaler@univ-rennes1.fr}





\maketitle

\section{Introduction}
 Let $A$ be an abelian surface over $\C$ with a principal polarisation $\lambda=\phi_{\theta}: \,A\rightarrow \hat{A}$, where $\theta \in \NS(A)$ is an ample divisor in the Néron-Severi group of $A$. Let $q_A$ be the integral quadratic form defined by half the intersection pairing $(\, . \,)$ on $\NS(A)$. Kani introduced another (positive definite) integral  quadratic form of smaller rank denoted $q_{(A,\lambda)}: \NS(A,\theta):=\NS(A)/\Z\theta \to \Z_{\geq0}$, and called it the \emph{refined Humbert invariant} associated to $(A,\lambda)$. It is defined as $q_{(A,\lambda)}(D)=(D.\theta)^2-2(D.D)$. If we denote by $\Ends(A,\lambda)$ the group of endomorphisms which are stable by the Rosati involution associated to $\lambda$, i.e. $a\mapsto a^\dagger:=\lambda^{-1}\hat{a}\lambda$, one has an isomorphism $\NS(A,\theta) \simeq \Ends(A,\lambda)/\Z \cdot\id$ defined by $D \mapsto \lambda^{-1} \phi_D=:a$. Under this identification,  $q_{(A,\lambda)}(a)=\Tr(a^2)-\Tr(a)^2/4$ where $\Tr()$ is the trace of the rational representation of $a$ (the result is stated in \cite[Rem.~16]{kani-formula} and proved over $\C$ in \cite[Cor.~9.2.3]{kirthesis}; see Section 2.3 in \cite{LGRV} for a proof over any field). 
 
 In \cite[Sec.~5]{kani-abelian}, Kani shows that $\Delta=q_{(A,\lambda)}(D)$ is the (classical) \emph{Humbert invariant} associated to $D$ in the sense of \cite{humbert}. When $0<\Delta \equiv 0,1 \pmod{4}$, it follows from the work of Humbert that $$H_{\Delta}:=\{(A,\lambda) \in \Ac_2(\C), \; \exists \text{ primitive } a \in\Ends(A,\lambda)/\Z \cdot\id:\; q_{(A,\lambda)}(a)=\Delta\}$$ is an irreducible hypersurface (see also \cite[Prop. IX.1.2]{geerhilbert}).  This motivated Kani to define the following loci.
\begin{definition}
   Let $(L_i,q_i)$ be two quadratic $\Z$-modules. We say that $(L_1,q_1)$ \emph{primitively represents} $(L_2,q_2)$  if there exists a linear injection $f : L_2 \to L_1$ such that 
   $q_1\circ f=q_2$ and $L_1/f(L_2)$ is torsion-free.
   If $q$ is an integral, positive-definite quadratic form on $\Z^r$, we define the \emph{Humbert locus} associated to $q$ as
   $$H(q):=\{(A,\lambda) \in \mathcal{A}_2(\C): \; q_{(A,\lambda)} \; \textrm{primitively represents} \; q\}.$$
\end{definition}
When the rank of $q$ is 1, i.e. $q(x)=\Delta x^2$, one has $H(q)=H_{\Delta}$ \cite[Sec.~2]{kanigeneralized}. More generally, there is a beautiful interplay between the geometry of the Humbert loci $H(q)$ and their intersections in $\Ac_2(\C)$ and the algebra and arithmetic of the associated quadratic forms. In a future article, we will give a general framework and review in more detail the various results in the literature. Here, we simply point to some references which were used during this work: this is far from exhaustive as for instance Kani has many more interesting articles on this topic. For rank 1 and the intersections of Humbert surfaces in special cases, see \cite{geer-lectures,geer-humbert,kani-abelian,gruenewald,kanigeneralized}; for several cases in rank 2, see  \cite{runge,rotgershimura, kanihumbert,guoyang, linyang,kani-survey}; for rank 3 \cite{kani_rank3, kir, kani_kir}. For an application to transcendence questions, see \cite{gijongomez}. Despite all this work, many aspects remain to be explored to get a complete characterisation of these loci (are they algebraic? Irreducible? What is their moduli description?) and efficient algorithms to manipulate them. To describe the intersections of Humbert surfaces in full generality for instance, one would need to look at Shimura curves with non-Eichler order. In the present article, we focus on the case where $q$ has rank 3, since, as far as we know, even a geometric characterisation of this set in $\Ac_2(\C)$ has never been given before. This is our main result.

    Let $q$ be a ternary quadratic form. Assume that $H(q) \ne \emptyset$ over the complex numbers and let  $(A,\lambda)$ be in $H(q)$.
    Albert's classification \cite[Prop~IX.1.2]{geerhilbert} implies that $A$ is isogenous to the square of a CM elliptic curve and a result from Shioda and Mitani (see for instance \cite[Cor.~10.6.3]{birkenhake})  implies that $A \simeq E_1 \times E_2$ with $E_1,E_2$ elliptic curves with CM-algebra $K=\Q(\sqrt{d_K})$ where $d_K<0$  is a fundamental discriminant.
    \begin{theorem} \label{th:main}
   Let $q$ be a ternary quadratic form over $\Z$ such that $H(q) \ne \emptyset$. Let us assume that the discriminant of $q$ is $-16 d_K$ where $d_K<0$ is a fundamental discriminant. Then $H(q)$ consists of the $\Gal(\bar{K}/K)$-orbit of a principally polarised abelian surface $(E_1 \times E_2,\lambda)$ with $E_1,E_2$ elliptic curves with CM by the maximal order $\Oc$ of an imaginary quadratic field $K$ of discriminant $d_K$. 
    \end{theorem}
\begin{remark}
    Necessary and sufficient conditions for $H(q)$ to be non-empty are given in \cite[Thm.~1]{kani_rank3} for the primitive CM case and in  \cite[Thm.~1.2]{kir} for the imprimitive CM case.
\end{remark}
 The condition on the discriminant of $q$ in Theorem~\ref{th:main} is equivalent to $E_1$ and $E_2$ having CM by the (same) maximal order $\Oc$. In a forthcoming article, we will drop the condition on the maximal order: if the two orders are the same, then the result (and most of the proof) will remain the same; however if the orders are distinct, then there may be more than one Galois orbit and most of the proof needs to be replaced. 
 
 Note that the action of the Galois group is not necessarily faithful and the cardinality of $H(q)$ can be smaller than $\# \Cl(\Oc)$: several elements in the orbit may actually be isomorphic as principally polarised abelian varieties, the extreme case of only one element is described in \cite{narbonne}. One may actually define the orbit as a dimension $0$ scheme where each point has the same multiplicity. Using \cite[Thm.~1]{kani_kir}, where a formula for the cardinality of $H(q)$ is given in the case where $(A,\lambda)$ is the Jacobian of a genus 2 curve $C$, we see that this multiplicity is equal to $\#\Aut(q_{(A,\lambda)})/\#\Aut(C)$ where $\Aut(q_{(A,\lambda)})=\{\alpha\in\Aut(\NS(A,\theta)):\,q_{(A,\lambda)}\circ\alpha=q_{(A,\lambda)}\}$. It would also be interesting to relate our work to  \cite[Sec.~3.2]{guoyang}, where principally polarised abelian surfaces with refined Humbert invariant of rank 3 (and their Galois orbits) are considered inside some fixed Shimura curves and where their singular relations are written explicitly.

The main result in Section \ref{sec:char} is Prop. \ref{prop:galconj}, where we prove that if $(A',\lambda') \in H(q_{(A,\lambda)})$, then $A'$ is a Galois conjugate of $A$. Our approach starts with \cite[Thm.~7]{kani_kir}, which shows that the equivalence class of the refined Humbert invariant determines the genus-equivalence class of a binary form associated with the ideal describing the isogenies between the two CM elliptic factors of $A$. This implies that these ideals are equal up to a square in a given class group, and consequently, that specific rank-2 modules are isomorphic. Moreover, by using the Artin map to link the action of $\Gal(\overline{K}/K)$ on elliptic curves to an ideal in the class group, we complete the proof by applying the results of \cite{kani_products}, which establish an equivalence of categories between the isomorphism classes of products of CM elliptic curves and the isomorphism classes of modules over imaginary quadratic orders.

The crux of proving Theorem~\ref{th:main} lies in Prop. \ref{prop:quo}, the proof of which occupies Sec.~\ref{sec:proof}. This proposition establishes a bijection between $\Cl(\mathcal{O})[2] \times \Z/2\Z$ and the finite quotient set $G_A/H_A$, which consists of isometries of $(\NS(A),q_A)$ that map polarisations to polarisations, modulo the natural action of $\Aut(A)$ by pullback. This quotient plays a central role: in \cite[Cor.~15]{kani_genus2}, Kani demonstrates that there is a natural surjective map from this quotient to the set of isomorphism classes of principally polarised abelian surfaces $(A,\mu) \in H(q)$ (for the $A$ which is fixed here). We construct explicit representatives for each coset in $G_A/H_A$. To achieve this, we first explicitly describe $\NS(A)$ as a $\Z$-module $\Ht(\mathcal{O})$ of (twisted) Hermitian matrices. We then provide an elementary proof that the direct isometries of the vector space $\Ht(\mathcal{O}) \otimes \Q$ can be represented as matrices in $\GL_2(K)$. By analysing their (twisted) content, we derive criteria ensuring that these matrices also stabilise $\Ht(\mathcal{O})$. The notion of (twisted) content, see eq. (\ref{eq:content}), is applied in this context for the first time, as far as we know, and is probably an interesting tool for further investigations. 

Finally, we conclude Sec.~\ref{sec:charAlambda} by proving Theorem \ref{th:main}. It is straightforward to see that the action of $\sigma \in \Gal(\overline{K}/K)$ on $(A,\lambda)$ preserves $q_{(A,\lambda)}$. Conversely, if $(A',\lambda') \in H(q_{(A,\lambda)})$, Prop.~\ref{prop:galconj} guarantees that $A' \cong A^{\sigma}$ for some $\sigma$. Thus, it only remains to classify the possible principal polarisations $\mu$ on $A$ such that $(A,\mu) \in H(q_{(A,\lambda)})$. The $\mu$'s are obtained via the action of the explicit representatives of $G_A/H_A$ constructed above. Using \cite{narbonne}, we can geometrically reinterpret the action of the ones coming from the $\Cl(\Oc)[2]$ part as the pullback of the polarisation $\lambda^{\sigma}$ on a Galois conjugate $A^{\sigma}$ isomorphic to $A$. We furthermore show that the action of the $\Z/2\Z$-factor is trivial. This concludes the proof.

Notice that using the results of \cite{narbonne} and the code of \cite{KNRR}, we have been able to experimentally confirm our theorem for fundamental discriminants up to $d_K=-199$.

\subsection{Acknowledgement} Both authors have been partially funded by the ANR project Melodia (ANR-20-CE40-0013) to carry out the research of the present work. The statements, the proof of Proposition~\ref{prop:galconj} and most of the strategy for Theorem~\ref{th:main} were designed by the authors without AI. But, the Gemini large language model (May 2026) introduced the idea of the `twisted content' in Sec.~\ref{sec:proof} and drafted the arguments for Lemma~\ref{lem:autherm}. The
authors checked and revised all of this output and take full responsibility for the
mathematical content. They also thank Ernst Kani and Harun K{\i}r for their careful proofreading and suggestions which greatly improved our first version. Claude Opus~5.5  (Anthropic) was used to check the final version of this paper for typos and small errors. 

\section{Characterising the abelian surfaces in $H(q)$} \label{sec:char}
Let $q$ be an integral quadratic form of rank $3$ such that $H(q) \ne \emptyset$ and let $(A,\lambda) \in H(q)$ be a principally polarised abelian surface over $\mathbb{C}$.  As already mentioned in the introduction, since $q$ is a ternary form, $NS(A)$ has rank 4 and Albert's classification and Shioda and Mitani's result  imply that $A \simeq E_1 \times E_2$ with $E_1,E_2$ elliptic curves with CM-algebra $K=\Q(\sqrt{d}_K)$ where $d_K<0$ is a fundamental discriminant. Moreover, $H(q)=H(q_{(A,\lambda)})$ because  the two lattices  $(\Z^3,q)$ and $(\Ends(A,\lambda)/\Z \cdot\id,q_{(A,\lambda)})$ are isometric since they have the same rank and one represents the other primitively.

\begin{proposition} \label{prop:galconj}
 If $(A',\lambda') \in H(q_{(A,\lambda)})$, then $A'$ is a $\Gal(\bar{K}/K)$-Galois conjugate of $A$.
\end{proposition}
\begin{proof}
    
    Let $(A',\lambda')$ be a principally polarised abelian surface with a refined Humbert ternary quadratic form $q_{(A',\lambda')}$ primitively representing $q_{(A,\lambda)}$. Albert's classification implies that this form has also rank 3, hence the two forms are $\GL_3(\Z)$-equivalent. Again with Shioda and Mitani's result, we also have $A' \simeq E_1' \times E_2'$ with $E'_1,E'_2$ elliptic curves with CM-algebra $K'=\Q(\sqrt{d_{K'}})$ where $d_{K'}$ is a fundamental discriminant. By \cite[Thm.~5]{kani_products}, we can actually assume that the conductor $f_1$ of $End(E_1)$ (resp. $f_1'$  of $E_1'$) divides the conductor $f_2$ of $E_2$ (resp. $f'_2$ of $E_2'$). Now  by \cite[Prop.~6]{kani_kir} and \cite[(75)]{kani_products}, we have that  \begin{multline*}
\disc(q_{(A,\lambda)}) = -16\disc(q_{E_1,E_2})
  = -16\operatorname{lcm}(f_1,f_2)^2 d_K = -16 f_2^2 d_K =\\
   \disc(q_{(A',\lambda')}) = -16\disc(q_{E_1',E_2'}) = -16 {f_2'}^2 d_{K'}
\end{multline*}
    where $q_{E_1,E_2}$ (resp. $q_{E_1',E_2'}$) is a representative of the ($\GL_2(\Z)$-class) of the binary quadratic form $\deg()$ on $\Hom(E_1,E_2)$ (resp. $\Hom(E'_1,E'_2)$).
    Hence, we have that $d_K=d_{K'}$ and $f_2=f_2'$. Again by \cite[Thm.~5]{kani_products}, we can then assume that $E_2=E_2'$ (this will simplify the rest of the proof).

    The proof of Theorem 7 in \cite{kani_kir} implies that the binary quadratic forms $q_{E_1,E_2}$ and $q_{E_1',E_2'}$ are $\GL_2(\Z)$-genus-equivalent. Hence, they have the same content. But \cite[Prop.~40]{kani_products} shows that the content is equal to $m:=f_2/f_1=f_2'/f_1'=:m'$ hence $f_1=f_1'$ as well. We can therefore write $q_{E_1,E_2} \sim_{\GL_2(\Z)} m \cdot q$ (resp. $q_{E_1',E_2'} \sim_{\GL_2(\Z)} m' \cdot q'$) where $q,q'$ are ($\GL_2(\Z)$-classes of) primitive quadratic forms with discriminant $f_1^2 d_K$. By \cite[Prop.~40]{kani_products}, $q=q_{\a}$ (resp. $q'=q_{\a'}$) where $\a$ (resp. $\a'$)  is a representative of the $\Oc_1:=\End(E_1)$-invertible ideal class $I_{E_2}(E_1)$ (resp. $I_{E_2}(E_1')$) (it is invertible in $\Oc_1$ because by \textit{loc.~cit.} formula (73), the multiplier ring of $\a$ and $\a'$ is $\Oc_1$).
    
    Since $q_{E_1,E_2}$ and $q_{E_1',E_2'}$ are $\GL_2(\Z)$-genus-equivalent with the same content, $q$ and $q'$ are also $\GL_2(\Z)$-genus-equivalent. In the usual group  of $\SL_2(\Z)$-equivalence classes of positive definite quadratic forms of discriminant $f_1^2 d_K$ (which is isomorphic to the class group $\Cl(\Oc_1)$), $\begin{pmatrix}
        1 & 0 \\ 0 & -1
    \end{pmatrix}$ induces the inverse in $\Cl(\Oc_1)$, hence preserves every genus of this group. By compositing locally with this action, we can assume that  $q$ and $q'$ are actually $\SL_2(\Z)$-genus-equivalent. 
    
    As explained in \cite[Rem.~41]{kani_products}, the $\Oc_1$-ideal classes associated to $q$ and $q'$ are represented by $\a^{\pm 1}$ and ${\a'}^{\pm 1}$, the ambiguity coming from the $\GL_2(\Z)$ vs. $\SL_2(Z)$ orbits. By the Principal Genus Theorem of Gauss  \cite[Thm.~3.15]{coxlib} applied to $q$ and $q'$, the primitive forms $q$ and $q'$ are $\SL_2(\Z)$-genus-equivalent if and only if there exists an invertible ideal $I_1 \subset \Oc_1$ such that ${\a'}^{\pm 1} \equiv I_1^2 {\a}^{\pm 1}$ in $\Cl(\Oc_1)$. Modifying $I_1$ if necessary, we can assume (and we  do) that $\a' \equiv I_1^2 \a$.  By \cite[Cor.~7.17]{coxlib}, we can then choose the ideal $I_1$ such that there exists $I_2$ in $\Oc_2:=\End(E_2)$ that is invertible and such that $\mu\cdot I_2\cdot\Oc_1=I_1$ for some $\mu\in K^*$. Since $\a=\a \Oc_1$, we get that the ideals $\a'$ and  $I_2^2 \a$ are equivalent modulo principal ideals.

 Let $K_2$ be  the ring class field of $\Oc_2$ and $\sigma \in \Gal(K_2/K)$ be  the image by the Artin map of the invertible ideal $I_2$ of $\Oc_2$. If $E_i \simeq \C/L_i$, then $E_i^{\sigma} \simeq \C/(I_2^{-1} L_i)$ by \cite[Chap.~10, Thm.~5]{langCM}. Using \cite[Cor.~34]{kani_products} (where for a lattice $L$, $E_L:=\C/L$), we then get that $$I_{E_2}(E_1^{\sigma})=I_{E_{L_2}}(E_{I_2^{-1} L_1}) \simeq L_2 (I_2^{-1} L_1)^{-1}= (L_2 L_1^{-1}) I_2 \simeq I_{E_2}(E_1) I_2= \a I_2.$$
By \cite[Thm.~3]{kani_products}, the $\Oc_2$-module  corresponding to $E_1 \times E_2$ is  $I_{E_2}(E_1) \oplus I_{E_2}(E_2) = I_{E_2}(E_1) \oplus \Oc_2$.
 Similarly, the $\Oc_2$-module corresponding to $E_1^{\sigma} \times E_2^{\sigma}$ is isomorphic to $ \a I_2 \oplus I_2$.  By \cite{borevich} (also \cite[Thm.~48]{kani_products}), this module is isomorphic to $\a I_2^2 \oplus \Oc_2 \simeq \a' \oplus \Oc_2= I_{E_2}(E_1') \oplus I_{E_2}(E_2)$. But this is the module characterising $E_1' \times E_2'$ since $E_2=E_2'$ and by \cite[Prop.~65]{kani_products} we get that $E_1' \times E_2' \simeq E_1^{\sigma} \times E_2^{\sigma}$.
\end{proof}

\section{Characterising the $(A,\lambda) \in H(q)$}\label{sec:charAlambda}
To prove Theorem~\ref{th:main}, we now restrict ourselves to the case where $A\simeq E_1 \times E_2$ with $E_1$ and $E_2$ elliptic curves with CM by the same maximal order $\Oc$ of an imaginary quadratic field $K$. Analytically, we can write $E_1 \simeq \C/\Oc$ and $E_2 \simeq \C/\b$ for an ideal $\b \subset \Oc$. Then, we have 
\[ \operatorname{End}(A) \simeq \begin{pmatrix} \Oc & \b^{-1} \\ \b & \Oc \end{pmatrix} = \left\{ \begin{pmatrix} a & b \\ c & d \end{pmatrix} \in M_2(K) \;\middle|\; a, d \in \Oc,\, b \in \b^{-1},\, c \in \b \right\}.
\]

Identifying $\hat{E}_1= \mathbb{C}/(\Oc/\sqrt{d_K})$  and $\hat{E}_2=\mathbb{C}/(\bar{\mathfrak{b}}^{-1}/\sqrt{d_K})$ by \cite[§~6.3]{shimura61:_compl_abelian} (see also \cite[Prop.~3.14]{gelin}; we will drop the normalisation factor $\sqrt{d_K}$), the principal product polarisation $\lambda_0 : A \xrightarrow{\sim} \hat{A}$ is uniquely represented by the diagonal matrix
\[
\lambda_0 = \begin{pmatrix} 1 & 0 \\ 0 & N(\b)^{-1} \end{pmatrix},
\]
where $N()$ is the norm of $K/\Q$.
The Rosati involution $\dag$ maps an endomorphism $P=\begin{pmatrix} a & b \\ c & d \end{pmatrix}$ to 
$$P^\dagger =\lambda_0^{-1} \bar{P}^t \lambda_0=\begin{pmatrix} \bar{a} & \frac{\bar{c}}{N(\mathfrak{b})} \\ \bar{b} N(\mathfrak{b}) & \bar{d} \end{pmatrix}.$$
We can therefore identify $\NS(A)$ with the sublattice $\Ht(\Oc):=\left\{S \in \begin{pmatrix}
    \Oc & \b^{-1} \\ \b & \Oc
\end{pmatrix}, \; S^{\dag}=S\right\}$ of $\End(A)$. Notice that $S \in \Ht(\Oc)$ if and only if  $$S=\begin{pmatrix} x & b \\ \bar{b} N(\b) & y \end{pmatrix} \; \textrm{with} \; x, y \in \mathbb{Z}, b \in \b^{-1}.$$  
Following the computations in \cite[Prop.~23]{kanihumbert} and noticing that $\deg(b)=[b^{-1}\O:\b]=N(b) N(\b)$, the quadratic form $q_A=\frac{1}{2} (.)$ on $\NS(A)$ agrees with $\det(S)$ on $\Ht(\Oc)$.
 
For the rest of the proof, we will use some of the notation and results of \cite{kani_genus2}, which, in the present context, can be made more explicit. In \textit{loc.~cit.} Prop.~8, the group $G_A$ is defined as the subgroup of the orthogonal group $\Aut(q_A):=O(\NS(A),q_A)\simeq O(\Ht(\Oc),\det)$ preserving  the ample cone, i.e. the set of all polarisations. Under this isomorphism,  the ample cone is formed by the elements $S \in \Ht(\Oc)$, such that $\lambda_0 S$ is positive definite, i.e. such that $S$ has positive trace and determinant.

The subgroup $H_A \subset G_A$ is defined as the image of the natural morphism from the geometric automorphism group $\Aut(A)=\{ P \in \End(A), \,  \det(P) \in \O^{\times}\}$ to $\Aut(q_A)$, acting on the elements of $\NS(A)$ by pullback. With the previous matrix representation in mind, this is $P\mapsto h_P(S)=P^{\dag} S P$. \\

We will now define some elements of $G_A$ and prove in Sec.~\ref{sec:proof} that they represent each coset $G_A/H_A$. Let $\a \subset \Oc$ be an ideal whose class is of 2-torsion in the class group $\Cl(\Oc)$. The two modules $\Oc \oplus \b$ and $\a \oplus \a \b$ are isomorphic by Steinitz \cite{steinitz} since $(1) \cdot \b \equiv \a \cdot \a \b$ in $\Cl(\O)$. Let $M_{\a}=\begin{pmatrix}
    \alpha & \beta  \\ \gamma & \delta
\end{pmatrix} \in \GL_2(K)$ be a matrix representation of an $\Oc$-module isomorphism $\Oc \oplus \b \xrightarrow{\sim} \a \oplus \a \b$. Since $\alpha \Oc + \beta \b = \a$ and $\gamma \Oc + \delta \b = \a \b$ we see that $\alpha \in \a$, $\beta \in \a \b^{-1}$, $\gamma \in \a \b$ and $\delta \in \a$ (this is going to be useful in the proof below). Moreover comparing the volumes of the lattices, we get that $N(\det(M_\a)) N(\Oc) N(\b) = N(\a) N(\a \b)$ hence $N(\det(M_{\a}))= N(\a^2)=N(\a)^2$.

We define $\gamma_{\a}$ acting on $\Ht(\Oc)$ by:
\begin{equation} \label{eq:gammaa}    
\gamma_{\a}(S) = \frac{1}{N(\a)} M_{\a}^\dagger S M_{\a},
\end{equation}
Since $\det(M_\a^\dag)=\overline{\det(M_\a)}$, we see that $\gamma_{\a}$ preserves the determinant and is therefore an isometry of the $\Q$-vector space $(\Ht(\Oc) \otimes \Q,\det)$.

\begin{remark}\label{rem:direct}
The map $S\mapsto\lambda_0 S$ identifies $\mathrm{Herm}^{\mathrm{twist}}_2(\mathcal O)\otimes\mathbb Q$
with the $\mathbb Q$-vector space of $2\times2$ Hermitian matrices over $K$, and
transforms $S\mapsto M^\dagger SM$ into $H\mapsto \bar M^{t}HM$. Checking on
diagonal and elementary matrices, one sees that the latter has determinant
$N(\det M)^2$ as a $\mathbb Q$-linear map. Hence, any isometry of the form
$S\mapsto c^{-1}M^\dagger SM$ (so $c^2=N(\det M)$) has determinant $1$ as a
$\mathbb Q$-linear map. 
\end{remark}
 
 \begin{lemma} One has $\gamma_{\a} \in G_A$ and it is a direct isometry.
 \end{lemma}
 \begin{proof} 
 We first show that $\gamma_\a$  preserves the $\Z$-module $\Ht(\Oc)$.
For $S \in \Ht(\Oc)$, we compute
 $$
 \gamma_{\a}(S)=\frac{1}{N(\a)}\begin{pmatrix}xN(\alpha)+\tr(\bar{\alpha}b\gamma)+yN(\gamma)/N(\b)& \bar{\alpha}x\beta+\bar{\alpha} b\delta+\bar{\gamma}\bar{b}\beta+\bar{\gamma}y\delta/N(\b)\\ {\alpha}x\bar{\beta}N(\b)+{\alpha} \bar{b}\bar{\delta}N(\b)+{\gamma}{b}\bar{\beta}N(\b)+{\gamma}y\bar{\delta} & N(\beta)N(\b)x+N(\b)\tr(\bar{\beta}b\delta)+yN(\delta)\end{pmatrix}.
 $$
 The diagonal elements are in $\Q$, we need to check that they are actually in $\Z$. The element in the position $(2,1)$ is $N(\b)$ times the conjugate of the element in the position $(1,2)$. We need to check that the latter is in $\b^{-1}$.
 Since $\alpha\in\a$ and $\beta\in\a\b^{-1}$, $ N(\a)\mid N(\alpha)$ and $N(\a)\mid N(\beta)N(\b)$. Similarly, we get $N(\a)\mid N(\delta)$ and $N(\a)N(\b)\mid N(\gamma)$. 
 Now, since $\bar{\alpha}\in\bar{\a}$ and $b\gamma\in\b^{-1}\a\b=\a$, we get that $\tr(\bar{\alpha}\gamma b)/N(\a)\in\Z$.
 Similarly, $N(\b)\tr(\bar{\beta}b\delta)/N(\a)\in\Z$. So the diagonal elements of $\gamma_{\a}(S)$ are integers.
 
 Since $\beta \b\subseteq\a$, we get $\bar{\alpha}\beta\b\subset N(\a)\O$, and then $\frac{\bar{\alpha}\beta}{N(\a)}\b\subseteq \O$ which gives $\frac{\bar{\alpha}x\beta}{N(\a)}\in \b^{-1}$. Similar arguments yield $\bar{\gamma}y\delta/(N(\b)N(\a))\in\b^{-1}$. Now, since $\alpha,\delta\in\a$ and $b\in\b^{-1}$, we get that $\frac{\bar{\alpha}\delta b}{N(\a)}\in \b^{-1}$. Finally, since $\beta\in \a\b^{-1}$ and $\gamma\in\a\b$, we conclude that $\frac{\bar{\gamma}\bar{b}\beta}{N(\a)}\in\b^{-1}$.
 So the right upper corner of $\gamma_{\a}(S)$ is in $\b^{-1}$. This proves that $\gamma_\a(\Ht(\Oc)) \subset \Ht(\Oc)$. Since $\gamma_\a$ is an isometry of $\Ht(\Oc)\otimes \Q$, it preserves the volume of the lattices and therefore $\gamma_\a(\Ht(\Oc))= \Ht(\Oc)$.
 
Thus, the map $\gamma_{\a}$ belongs to $\Aut(q_A)$. In addition, it also preserves the ample cone: as noted in \cite[Prop.~8]{kani_genus2}, it is enough to compute for instance $\gamma_{\a}(\textrm{Id}_2)$ and to observe by explicit computation that its trace is positive. Finally, it is a direct isometry because of Rmk.~\ref{rem:direct}.
  
  \end{proof}
 \begin{remark}
     As we will see later, we can interpret $\gamma_\a(S)$ as the pullback of a polarisation on an isomorphic Galois conjugate of $A$ and this gives another proof that it satisfies all the properties.
 \end{remark}

Notice that if $\a'= \mu \a$, $\mu \in K^*$ is another representative of the same ideal class, then $\mu M_{\a}$ defines an isomorphism $\O \oplus \b \to \a' \oplus \a' \b$, which still induces the same $\gamma_{\a}$. Finally, two isomorphisms $\O \oplus \b \to \a \oplus \a \b$ differ by an element of $\Aut(A)$, hence we have a well defined map $\Cl(\O)[2] \to G_A/H_A$ associating to the ideal class of $\a$, the coset $\gamma_\a H_A$.

Finally, we define an involution $\tau \in G_A$ by $$\tau\left(\begin{pmatrix}
    x & b \\ \bar{b}N(\b) & y
\end{pmatrix}\right) = \begin{pmatrix}
    y & b \\ \bar{b}N(\b) & x 
\end{pmatrix}.$$ 

\begin{proposition} \label{prop:quo}
The $2 \cdot \# \Cl(\Oc)[2]$ elements $\{\gamma_{\mathfrak{a}}\}$ and $\{\gamma'_{\a} := \tau \circ \gamma_{\a}\}$ where $\a$ spans a set of representatives of $\Cl(\Oc)[2]$ form a complete set of pairwise distinct representatives for $G_A/H_A$. In particular $\# G_A/H_A=2 \# \Cl(\Oc)[2]$ as mentioned in  \cite[(8)]{kani_genus2}. 
\end{proposition}
\begin{remark}
    It is tempting to conclude that $G_A/H_A$ is a group isomorphic to a semi-direct product of $\Cl(\Oc)[2]$ and $\Z/2\Z$ but, as we will see in the proof of the result, an essential element, the twisted content of $M_{\a}$ (introduced in \eqref{eq:content}), is likely not a morphism in general and we cannot easily conclude. 
\end{remark}
The proof is postponed to Sec.~\ref{sec:proof}. Let us show how this allows to conclude the proof of Theorem~\ref{th:main}.

\begin{proof}[Proof of Theorem~\ref{th:main}] The Galois action clearly preserves the characteristic polynomial of an endomorphism of an abelian variety and its Néron-Severi group, so the refined Humbert invariants of two Galois conjugate principally polarised abelian varieties are $\GL_3(\Z)$-equivalent.

  Conversely, by Prop.~\ref{prop:galconj}, if $(A',\lambda') \in H(q)$, we have $A'\simeq E_1^{\sigma} \times E_2^{\sigma}$ for $\sigma \in \Gal(\bar{K}/K)$. The Galois action induces a bijection between the sets, denoted  $\bar{\mathcal{P}}(A,q)$ and $\bar{\mathcal{P}}(A',q)$ in \cite{kani_genus2}, of isomorphism classes of principal polarisations on $A$ and $A'$ with a given refined Humbert invariant $q$. It is therefore enough to show that for any $\mu \in \bar{\mathcal{P}}(A,q)$, there exists $\sigma \in \Gal(\bar{K}/K)$ such that $(A,\mu) \simeq (A^{\sigma},\lambda^{\sigma})$.
  
    As proved by \cite[Cor.~15]{kani_genus2}, $\bar{\mathcal{P}}(A,q)$ is isomorphic to the double quotient $\Sc_{\lambda} \setminus G_A / H_A$ where $\Sc_{\lambda}=\{g \in G_A, \; g(\lambda)=\lambda\}$, in particular $\{g(\lambda), \; g \in G_A/H_A\}=\bar{\mathcal{P}}(A,q)$. By Prop.~\ref{prop:quo}, there exists $g \in \bigcup_{\a \in \textrm{Cl}(\Oc)[2]} \{\gamma_{\a}\} \cup \{\gamma'_{\a}\}$ such that  $\mu=g(\lambda)$. It remains to see how $g$ is linked to the Galois action on $\lambda$.
    
 The polarisation $\lambda \in NS(A)$ corresponds to a matrix $S_0=\left(\begin{smallmatrix} x & b \\ \bar{b}N(\mathfrak{b}) & y \end{smallmatrix}\right) \in \Ht(\Oc)$ with determinant 1 and positive trace. Let $H$ be the Hilbert class field of $K$. By \cite[Thm.~4]{narbonne}, for any $\sigma \in \Gal(H/K)$ associated to the class of an ideal $\a^{-1} \in \Oc$, the polarisation $\lambda^{\sigma}$ on $E_1^{\sigma} \times E_2^{\sigma}$ corresponds to the matrix $\frac{1}{N(\a)} S_0$ for the module $\a \oplus \a \b$. When the class of $\a$ is of 2-torsion, we have defined an isomorphism $M_\a : \Oc \oplus \b \to  \a \oplus \a \b$, i.e. an isomorphism $\phi_\a : A \to A^{\sigma}$. We therefore have the two corresponding commutative diagrams
 \begin{center}
    \begin{tikzcd}[column sep=2cm]
        \Oc \oplus \b \arrow[r, "M_\a"] \arrow[d, "\lambda_0 \gamma_\a(S_0)"'] & \a \oplus \a \b \arrow[d, "\lambda_0 \frac{1}{N(\a)} S_0"] \\
        \Oc \oplus \bar{\b}^{-1}  & \bar{\a}^{-1} \oplus  \bar{\a}^{-1} \bar{\b}^{-1} \arrow[l, "\lambda_0 M_a^\dag \lambda_0^{-1}"]
    \end{tikzcd}
    \hspace{2cm} 
    \begin{tikzcd}
        A \arrow[r, "\phi_\a"] \arrow[d, "\phi_\a^{*}(\lambda^{\sigma})"'] & A^{\sigma} \arrow[d, "\lambda^{\sigma}"] \\
        \hat{A}  & \hat{A^{\sigma}} \arrow[l, "\hat{\phi}_\a"]
    \end{tikzcd}
\end{center}
and this gives the geometric interpretation of the polarisation $\gamma_\a(S_0)$ as the pullback of $\lambda^\sigma$. 
  
    It remains to deal with  the geometric description of the action of $\tau \in G_A$ on $S_0$. Let $$P=\begin{pmatrix}
        y & b \\ -\bar{b} N(\b) & -x 
    \end{pmatrix} \in \GL_2(K), \; \textrm{where } x,y,b \; \textrm{are the entries of } S_0.$$ Because $S_0$ is the matrix of a principal polarisation, $\det(S_0)=xy-b \bar{b} N(\b)=1$, hence $\det(P)=-1 \in \O^{\times}$. Moreover, $P \in \End(A)$, thus $P$ is actually an automorphism of $A$. An explicit computation shows that $P^{\dag} S_0 P = \tau(S_0)$.\footnote{This does not contradict the fact that $\tau \notin G_A^+$ (with the notation of Sec.~\ref{sec:proof}) because the matrix $P$ depends on $S_0$.} Hence the isomorphism class of the polarisation $\tau(S_0)$ is the same as the one of $S_0$ and the action of $\tau$ is superfluous. 
    
\end{proof}
\begin{remark}
    This proof shows in particular that $\# \bar{\Pc}(A,q) \leq \#\Cl(\Oc)[2]$. The fact that $\tau$ does not play a role looks a bit magical. It is actually related (by a long detour through various isomorphisms) to \cite[Prop.~10]{narbonne} where it is shown that the action of the complex conjugation is also superfluous with respect to the action of $\Gal(H/K)$.
\end{remark}

\section{Proof of Proposition~\ref{prop:quo}} \label{sec:proof}
We split the proof in three steps.
    
\subsection{Step 1: Decomposition via orientation}
The lattice $\Ht(\Oc)$ is a free $\mathbb{Z}$-module of rank 4. Any isometry $g \in G_A$ induces a $\mathbb{Z}$-linear transformation on $\Ht(\Oc)$, which can be represented by a $4 \times 4$ integer matrix with determinant that we denote in this representation $\det_4(g) = \pm 1$. 

Let $G_A^+ := \{ g \in G_A : \det_4(g) = +1 \}$ be the subgroup of direct isometries. Choosing a standard basis for $\Ht(\Oc)$:
\[
e_1 = \begin{pmatrix} 1 & 0 \\ 0 & 0 \end{pmatrix}, \quad e_2 = \begin{pmatrix} 0 & 0 \\ 0 & 1 \end{pmatrix}, \quad e_3 = \begin{pmatrix} 0 & w_1 \\ \bar{w_1} N(\b) & 0 \end{pmatrix}, \quad e_4 = \begin{pmatrix} 0 & w_2 \\ \bar{w_2} N(\b) & 0 \end{pmatrix},
\]
where $\b^{-1}= \Z w_1 \oplus \Z w_2$, we see that $\tau(e_1)=e_2$, $\tau(e_2)=e_1$, $\tau(e_3)=e_3$, and $\tau(e_4)=e_4$. Thus, $\tau$ acts as a reflection on a $4$-dimensional space, implying $\det_4(\tau) = -1$. Consequently, $\tau \notin G_A^+$, which yields the coset decomposition:
\[
G_A = G_A^+ \sqcup \tau G_A^+.
\]

\subsection{Step 2: defining an injection from $G_A^+/H_A$ to $\Cl(\Oc)[2]$}
Every $g \in G_A^+$ defines a direct isometry  of $(\Ht(\Oc),\det)$ preserving the positivity of the trace. The following lemma is probably well-known and obtained as a particular case of results on Clifford algebras. We give here an elementary proof in the case of (twisted) Hermitian matrices.

\begin{lemma} \label{lem:autherm}
Let $K$ be an imaginary quadratic field and $\b$ an  ideal in its maximal order $\Oc$. Let
$\textrm{Herm}_2^{\textrm{twist}}(K) = \{ S \in M_2(K) \mid S^\dagger = S \}$ be the $\Q$-vector space containing the lattice $\textrm{Herm}_2^{\textrm{twist}}(\Oc)$
equipped with the quadratic form $\det$. 
Then, the special orthogonal group  $\textrm{SO}(\textrm{Herm}_2^{\textrm{twist}}(K))$ is exhaustively described by the action of $\operatorname{GL}_2(K)$, given by:
\[
g(S) = \frac{1}{c}  M^\dagger S M
\]
for some $M \in \operatorname{GL}_2(K)$ and $c \in \mathbb{Q}^\times$.
\end{lemma}

\begin{proof}
The symmetric bilinear form associated with the quadratic form $S \mapsto \det(S)$ is $B(S, T) = \frac{1}{2}(\det(S+T) - \det(S) - \det(T))$.
The adjoint matrix  $\tilde{S}$ of $S$ satisfies $\tilde{S} = \operatorname{tr}(S)\id_2 - S$. The space $\textrm{Herm}_2^{\textrm{twist}}(K)$ is stable under the operator $\tilde{\phantom{s}}$. For any $S,T \in \textrm{Herm}_2^{\textrm{twist}}(K)$, using the fact that $\det(S+T) \id_2=(S+T)\tilde{(S+T)}$, we obtain that
\begin{equation} \label{eq:adjugate}
S\tilde{T} + T\tilde{S} = \operatorname{tr}(S\tilde{T})\id_2 = 2 B(S,T)\id_2.
    \end{equation}

Let $S \in \textrm{Herm}_2^{\textrm{twist}}(K)$ be a non-isotropic vector (i.e., $\det(S) \neq 0$). The orthogonal reflection $\rho_S$ across the hyperplane orthogonal to $S$ is given by:
\[
\rho_S(T) = T - \frac{2 B(S,T)}{\det(S)}S.
\]
Using \eqref{eq:adjugate}, we can rewrite this geometric reflection purely in terms of matrix multiplication:
\[
\rho_S(T) = T - \frac{1}{\det(S)}(S\tilde{T} + T\tilde{S})S = T - \frac{1}{\det(S)}(S\tilde{T}S + T(\det S)\id_2) = - \frac{1}{\det(S)} S\tilde{T}S.
\]

By the Cartan-Dieudonn\'e Theorem, every isometry in the special orthogonal group $SO(\textrm{Herm}_2^{\textrm{twist}}(K))$ is the product of an even number of reflections. Let us compute the composition of two reflections $\rho_{S_1}$ and $\rho_{S_2}$ applied to $T$:
\begin{align*}
(\rho_{S_2} \circ \rho_{S_1})(T) &= - \frac{1}{\det(S_{2})} S_{2} \left( \widetilde{-\frac{1}{\det(S_{1})} S_{1}\tilde{T}S_{1}} \right) S_{2} \\
&= \frac{1}{\det(S_{1}S_{2})} S_{2} \left( \tilde{S_{1}} \tilde{\tilde{T}} \tilde{S_{1}} \right) S_{2} \\
&= \frac{1}{\det(S_{1}S_{2})} (S_{2}\tilde{S_{1}}) T (\tilde{S_{1}}S_{2}).
\end{align*}

Define $M = \tilde{S_{1}}S_{2} \in \operatorname{GL}_2(K)$ then 
$M^\dagger = (\tilde{S_{1}}S_{2})^\dagger = S_{2}^\dagger \tilde{S_{1}}^\dagger = S_{2}\tilde{S_{1}}$, hence
\[
(\rho_{S_2} \circ \rho_{S_1})(T) = \frac{1}{\det(S_1S_2)} M^\dagger T M.
\]

Since $\textrm{SO}(\textrm{Herm}_2^{\textrm{twist}}(K))$ is generated by products of pairs of reflections, every element of $\textrm{SO}(\textrm{Herm}_2^{\textrm{twist}}(K))$ is of this form. 
\end{proof}
\begin{remark}
    A direct computation with a basis of $\Ht(K)$ shows that the only matrices fixing every element of $\Ht(K)$ up to a scalar are the scalar matrices. Hence $g$ actually defines $M$  uniquely up to a scalar multiple in $K^*$. 
\end{remark}
\begin{remark}
By Remark~\ref{rem:direct}, $H_A\subseteq G_A^+$ and
$\gamma_{\mathfrak a}\in G_A^+$.
\end{remark}

Using this lemma, we can associate to $g \in G_A^+$ a matrix $M \in \operatorname{GL}_2(K)$ and a rational scalar $c \in \mathbb{Q}_{>0}$ such that:
\[
g(S) = \frac{1}{c} M^\dagger S M, \quad \forall S \in \Ht(\Oc).
\]
Since $g$ is an isometry we have $c^2=N(\det(M))$.

We now give a necessary condition for a direct isometry of $\Ht(K)$ to preserve the lattice $\Ht(\Oc)$. Writing $M= \begin{pmatrix} \alpha & \beta \\ \gamma & \delta \end{pmatrix}$, we define the \textit{twisted content} of $M$, denoted $\cont(M)$, as the fractional ideal:
\begin{equation} \label{eq:content}
\mathfrak{a} = (\alpha) + \beta\b + \gamma \b^{-1} + (\delta).    
\end{equation}

\begin{lemma} \label{lem:content}
We resume with the notation of Lemma~\ref{lem:autherm}. Let  $g : \Ht(K) \to \Ht(K)$ be a direct isometry given by:
\[
g(S) = \frac{1}{c} M^\dagger S M
\]
for some $M  \in \operatorname{GL}_2(K)$ and $c = \sqrt{N(\det (M))} \in \mathbb{Q}_{>0}$. 
Assume that $g$ preserves the lattice $\Ht(\Oc)$. Let $\a$ be the twisted content of $M$. Then the ideal class of $\a$ belongs to $\Cl(\O)[2]$. More precisely $\a^2=(\det(M))$.
\end{lemma}

\begin{proof} We start by proving that
$\mathfrak{a}\bar{\mathfrak{a}} \subseteq c\Oc.$
We evaluate $g$ on the diagonal basis elements of $\Ht(\Oc)$,  $\left(\begin{smallmatrix} 1 & 0 \\ 0 & 0 \end{smallmatrix}\right)$ and $\left(\begin{smallmatrix} 0 & 0 \\ 0 & 1 \end{smallmatrix}\right)$. Since their images belong to $\Ht(\Oc)$ by hypothesis, we get successively the inclusions (i-iii) and (iv-vi)
\begin{enumerate}
    \item $\alpha\bar{\alpha} \in c\mathbb{Z} \implies (\alpha)(\bar{\alpha}) \subseteq c\Oc$.
    \item $\bar{\alpha}\beta \in c\b^{-1} \implies (\beta\b) (\bar{\alpha}) \subseteq c\Oc \implies (\bar{\beta}\bar{\b}) ({\alpha}) \subseteq c\Oc $.
    \item $\beta \bar{\beta} N(\b)\in c \mathbb{Z} \implies (\beta \b) (\overline{\beta \b}) \subseteq c\Oc$.
     \item $\delta\bar{\delta} \in c\mathbb{Z} \implies (\delta)(\bar{\delta}) \subseteq c\Oc$.
    \item $\frac{\gamma\bar{\gamma}}{N(\b)} \in c\mathbb{Z} \implies (\gamma\b^{-1})(\bar{\gamma} \overline{\b^{-1}}) \subseteq c\Oc $.
   \item $\bar{\delta} \gamma \in c\b \implies (\gamma \b^{-1}) (\bar{\delta}) \subseteq c\Oc \implies (\bar{\gamma} \bar{\b}^{-1}) ({\delta}) \subseteq c\Oc$.
\end{enumerate}
Similarly, by evaluating $g$ on the off-diagonal elements $\left(\begin{smallmatrix} 0 & w \\ \bar{w} N(\b) & 0 \end{smallmatrix}\right)$ for $w \in \b^{-1}$, we get, as explained below, the inclusions
\begin{enumerate}[resume]
    \item $\bar{\alpha}w\gamma+\bar{\gamma}\bar{w}\alpha \in c\mathbb{Z} \implies (\bar{\alpha}\gamma)\b^{-1},(\alpha\bar{\gamma})\bar{\b}^{-1}\subseteq c\O$.
    \item $\bar{\alpha}w\delta+\bar{\gamma}\bar{w}\beta \in c\b^{-1} \implies (\bar{\alpha}\delta),(\bar{\gamma}\beta)\bar{\b}^{-1}\b\subseteq c\O\implies ({\alpha}\bar{\delta}),({\gamma}\bar{\beta})\b^{-1}\bar{\b}\subseteq c\O$.
     \item $(\bar{\beta}w\delta+\bar{\delta}\bar{w}\beta) N(\b)\in c \mathbb{Z} \implies (\bar{\beta}\delta)\bar{\b},(\bar{\delta}\beta)\b\subseteq c\O$.
\end{enumerate}
We give the details for $(\alpha\bar{\gamma})\bar{\b}^{-1}\subseteq c\O$ in (vii), the other cases follow in the same manner. From  $\bar{\alpha}w\gamma+\bar{\gamma}\bar{w}\alpha \in c\mathbb{Z}$ multiplied by $\epsilon\in\O$ and applied to $\epsilon w$, we get $(\epsilon-\bar{\epsilon})\bar{w}\bar{\gamma}\alpha\in c\O$, so $\bar{\gamma}\bar{\b}^{-1}\alpha\subseteq  c\mathcal{D}^{-1}$ where $\mathcal{D}$ is the different ideal of $\O$. The norm of $\frac{\bar{\gamma}\bar{\b}^{-1}\alpha}{c}$ being an integer by (i) and (v) implies that it is actually an integral ideal because at a ramified prime $\frak{p}\mid p$, $\frak{p}$ is the only prime above $p$ and has norm $p$). Hence,  the conclusion follows. 

Thus, if $\mathfrak{c}_1,\ldots,\mathfrak{c}_4$ denote the 4 ideals on the right hand side of \eqref{eq:content}, then the above equations (i) – (ix) show that $\mathfrak{c}_i\mathfrak{c}_j \subset  c \Oc$, for all $1 \leq i,j \leq 4$, and hence it follows
that $\bar{\a} \a \subset  c \Oc$, as desired. Hence, $N(\mathfrak{a}) \ge c$.

Let us now look at $\det(M) = \alpha\delta - \beta\gamma$. 
By definition of the twisted content $\mathfrak{a}$:
\begin{itemize}
    \item $\alpha \in \a$ and $\delta \in \a$, which implies their product $\alpha\delta \in \a^2$.
    \item $\beta \in \a\b^{-1}$ and $\gamma \in \a\b$, which implies their product $\beta\gamma \in (\a\b^{-1})(\a\b) = \mathfrak{a}^2$.
\end{itemize}
Thus, $\det (M) \in \mathfrak{a}^2$ and therefore 
$c^2= N(\det (M)) \geq N(\mathfrak{a}^2)$.
Since we  have seen that  $N(\mathfrak{a}) \ge c$, this implies $(\det (M)) = \mathfrak{a}^2$, i.e. the class of $\a$ is a 2-torsion element of $\textrm{Cl}(\Oc)$.
\end{proof}

\begin{lemma} \label{lem:image}
Let $M=\begin{pmatrix}
    \alpha & \beta \\ \gamma & \delta
\end{pmatrix} \in \GL_2(K)$ be a matrix with twisted content $\a:=\cont(M)$ such that $(\det(M))=\a^2$. Then $M(\Oc \oplus \b) = \a \oplus \a \b$. This equality shows that $\cont(M)$ is actually the unique ideal such that $M(\Oc \oplus \b) = \cont(M)(\Oc \oplus \b)$.
\end{lemma}
\begin{proof}
  The resulting $\Oc$-module satisfies $M ( \Oc\oplus  \mathfrak{b} )\subseteq (\alpha\Oc + \beta\mathfrak{b})\oplus ( \gamma\Oc + \delta\mathfrak{b})$.
 By definition of ideal addition, we immediately obtain the strict inclusions:
\begin{itemize}
    \item $\alpha \in \mathfrak{a}$ and $\beta\b  \subseteq \mathfrak{a} \implies \alpha\Oc + \beta\mathfrak{b} \subseteq \mathfrak{a}$.
    \item $\gamma\b^{-1} \subseteq \a$ and $\delta \in \a \implies \gamma\Oc + \delta\mathfrak{b} \subseteq \mathfrak{a}\mathfrak{b}$.
\end{itemize}
This implies that $M(\Oc \oplus \b) \subseteq \mathfrak{a} \oplus \mathfrak{a}\mathfrak{b}$. 

To prove equality, we compare the volumes (or absolute norms) of these lattices. The norm of the target module is $N(\mathfrak{a}) \cdot N(\mathfrak{a}\mathfrak{b}) = N(\mathfrak{a})^2 N(\mathfrak{b})$. 
On the other hand, the volume scaling factor of the linear map $M$ is given by $N(\det (M))= N(\mathfrak{a})^2$. The norm of the image lattice is therefore $N(\det (M)) \cdot N(\Oc \oplus \b) = N(\mathfrak{a})^2 N(\mathfrak{b})$. This shows the equality.
\end{proof}

Recall that we started with $g \in G_A^+$ and thanks to Lemmas~\ref{lem:autherm} and \ref{lem:content}, we have a matrix $M \in \GL_2(K)$, defined uniquely up to a scalar, such that $g(S)=\frac{1}{c} M^{\dag} S M$. The matrix $M$ has twisted content $\a$ with $\a^2=(\det (M))$. This therefore induces a well-defined map $\phi: G_A^+ \to \Cl(\Oc)[2]$. If $P \in \Aut(A)$, it stabilises $\Oc \oplus \b$, hence $MP(\Oc \oplus \b)=M(\Oc \oplus \b)$ and by Lemma~\ref{lem:image}, $\cont(MP)=\cont(M)$. The twisted content is invariant on $M\cdot\Aut(A)$. Thus, the map $\phi$ induces a map $\tilde{\phi} : G_A^+/H_A \to \Cl(\Oc)[2]$.
\begin{lemma}
The map $\tilde{\phi}$ is injective.    
\end{lemma}

\begin{proof}
    Let $g,g' \in G_A^+$ with associated matrices $M,M' \in \GL_2(K)$ and twisted contents in the same class of $\Cl(\Oc)[2]$. Because $M,M'$ are only defined up to a scalar, we can actually assume that $M$ and $M'$ have the same twisted content $\a$. By Lemma~\ref{lem:image}, since $M,M'$ realise isomorphisms between the same modules, the matrix $P = M^{-1} M'$ stabilises the module $\Oc \oplus \b$ so $P \in \End(A)$. Since $(\det M)=(\a^2)=(\det M')$, determinant of $P$ is a unit hence $P \in \Aut(A)$.
By substituting $M' = M P$ into the algebraic expression of the isometry $g'$, we obtain: 
\begin{align*}
g'(S) &= \frac{1}{c'} {M'}^\dagger S M' \\
&= \frac{1}{c \sqrt{N(\det(P))}} (M P)^\dagger S (M P) \\
&= P^\dagger \left( \frac{1}{c} M^\dagger S M \right) P \\
&= P^\dagger g(S) P.
\end{align*}
Therefore, $g' = h_P \circ g$, i.e. $g$ and $g'$ define the same coset in $G_A^+/H_A$.
\end{proof}

\subsection{Step 3: Conclusion}
For each representative $\a$ of a $2$-torsion ideal class, we have defined $\gamma_{\a} \in G_A^+/H_A$ in \eqref{eq:gammaa} through its matrix $M_{\mathfrak{a}}$. Since $\alpha \Oc + \beta \b = \a$ and $\gamma \Oc + \delta \b = \a \b$, $\cont(M_{\a}) = \a$ so the injective map $\tilde{\phi}$ is actually a bijection.
Thus, the sets $\{ \gamma_{\a} \}$ and $\{ \tau \circ \gamma_{\a} \}$ are completely disjoint modulo $H_A$, their elements are pairwise non-equivalent and form a complete set of representatives of $G_A/H_A$. In particular $\#(G_A/H_A)= 2 \cdot \#\Cl(\Oc)[2]$.

\bibliographystyle{alpha}
\bibliography{synthbib}

@incollection{kani-survey,
 author = {Kani, Ernst},
 title = {Subcovers of curves and moduli spaces},
 booktitle = {Geometry at the frontier. Symmetries and moduli spaces of algebraic varieties. 2016--2018 workshops, Universidad de la Frontera, Puc\'on, Chile, November 2016, November 2017, November 2018},
 isbn = {978-1-4704-5327-5; 978-1-4704-6422-6},
 pages = {229--250},
 year = {2021},
 publisher = {Providence, RI: American Mathematical Society (AMS)},
 language = {English},
 doi = {10.1090/conm/766/15384},
 zbMATH = {7370484},
 Zbl = {1464.14028}
}

@book{geer-lectures,
 author = {Hirzebruch, Friedrich and van der Geer, Gerard},
 title = {Lectures on {Hilbert} modular surfaces. {Notes} by {W}. {Hausmann} and {F}. {J}. {Koll}},
 fseries = {S{\'e}minaire de Math{\'e}matiques Sup{\'e}rieures. S{\'e}minaire Scientifique OTAN (NATO Advanced Study Institute)},
 series = {S{\'e}min. Math. Sup{\'e}r., S{\'e}min. Sci. OTAN (NATO Adv. Study Inst.).},
 volume = {77},
 year = {1981},
 publisher = {Les Presses de l'Universit{\'e} de Montr{\'e}al, Montr{\'e}al, Qu{\'e}bec},
 language = {English},
 zbMATH = {3758493},
 Zbl = {0483.14009}
}

@article{kani-formula,
 author = {Kani, Ernst},
 title = {Elliptic subcovers of a curve of genus 2. {I}: {The} isogeny defect},
 fjournal = {Annales Math{\'e}matiques du Qu{\'e}bec},
 journal = {Ann. Math. Qu{\'e}.},
 issn = {2195-4755},
 volume = {43},
 number = {2},
 pages = {281--303},
 year = {2019},
 language = {English},
 doi = {10.1007/s40316-018-0105-6},
 zbMATH = {7101858},
 Zbl = {1435.14029}
}

@article{steinitz,
 author = {Steinitz, Ernst},
 title = {Zur {Theorie} der {Moduln}.},
 fjournal = {Mathematische Annalen},
 journal = {Math. Ann.},
 issn = {0025-5831},
 volume = {52},
 pages = {1--57},
 year = {1899},
 language = {German},
 doi = {10.1007/BF01445346},
 url = {https://eudml.org/doc/157931},
 zbMATH = {2667517},
 JFM = {30.0097.01}
}

@book{langCM,
 author = {Lang, Serge},
 title = {Elliptic functions. {Second} edition},
 fseries = {Graduate Texts in Mathematics},
 series = {Grad. Texts Math.},
 issn = {0072-5285},
 volume = {112},
 isbn = {0-387-96508-4},
 year = {1987},
 publisher = {Springer‑Verlag, New York},
 language = {English},
 zbMATH = {3995866},
 Zbl = {0615.14018}
}

@book{birkenhake,
 author = {Birkenhake, Christina and Lange, Herbert},
 title = {Complex abelian varieties},
 edition = {2nd augmented},
 fseries = {Grundlehren der Mathematischen Wissenschaften},
 series = {Grundlehren Math. Wiss.},
 issn = {0072-7830},
 volume = {302},
 isbn = {3-540-20488-1},
 year = {2004},
 publisher = {Berlin: Springer},
 language = {English},
 zbMATH = {2120946},
 Zbl = {1056.14063}
}

@article{narbonne,
 author = {Narbonne, Fabien},
 title = {Polarized products of elliptic curves with complex multiplication and field of moduli {{\(\mathbb{Q}\)}} (with an appendix by {Francesc} {Fit{\'e}} and {Xavier} {Guitart})},
 fjournal = {Polynesian Journal of Mathematics},
 journal = {Polyn. J. Math.},
 issn = {3075-3422},
 volume = {1},
 pages = {32},
 year = {2024},
 language = {English},
 doi = {10.69763/polyjmath.1.5},
 zbMATH = {8194668}
}

@article{geer-humbert,
 author = {van der Geer, Gerard},
 title = {On the geometry of a {Siegel} modular threefold},
 fjournal = {Mathematische Annalen},
 journal = {Math. Ann.},
 issn = {0025-5831},
 volume = {260},
 pages = {317--350},
 year = {1982},
 language = {English},
 doi = {10.1007/BF01461467},
 url = {https://eudml.org/doc/163672},
 zbMATH = {3741555},
 Zbl = {0473.14017}
}

@misc{LGRV,
 author = {Elisa {Lorenzo Garc{\'{\i}}a} and Christophe Ritzenthaler and Fernando Rodr{\'{\i}}guez Villegas},
 title = {An arithmetic intersection for squares of elliptic curves with complex multiplication},
 year = {2024},
 howpublished = {Preprint, {arXiv}:2412.08738 [math.{NT}]},
 url = {https://arxiv.org/abs/2412.08738},
 arXiv = {arXiv:2412.08738}
}

@article{kir,
 author = {K{\i}r, Harun},
 title = {The classification of the refined {Humbert} invariant for curves of genus 2},
 fjournal = {International Journal of Number Theory},
 journal = {Int. J. Number Theory},
 issn = {1793-0421},
 volume = {21},
 number = {6},
 pages = {1247--1279},
 year = {2025},
 language = {English},
 doi = {10.1142/S1793042125500654},
 zbMATH = {8054356},
 Zbl = {1567.11073}
}

@article{kani_products,
 author = {Kani, Ernst},
 title = {Products of {CM} elliptic curves},
 fjournal = {Collectanea Mathematica},
 journal = {Collect. Math.},
 issn = {0010-0757},
 volume = {62},
 number = {3},
 pages = {297--339},
 year = {2011},
 language = {English},
 doi = {10.1007/s13348-010-0029-1},
 zbMATH = {5968124},
 Zbl = {1237.11027}
}

@misc{kani_rank3,
 author = {Kani, Ernst},
 title = {The {R}efined {H}umbert {I}nvariant for
Abelian Product Surfaces with Complex
Multiplication},
 year = {2021},
 note = {https://mast.queensu.ca/~kani/papers/},
}

@misc{kani_genus2,
 author = {Kani, Ernst},
 title = {Curves of Genus 2 on Abelian Surfaces},
 year = {2025},
 note = {https://mast.queensu.ca/~kani/papers/},
}

@misc{kani_kir,
 author        = {Kani, Ernst and K{\i}r, Harun},
 title         = {The Number of Curves of Genus 2 with a Given Refined {H}umbert Invariant},
 year          = {2026},
 eprint        = {2608.24122},
 archivePrefix = {arXiv},
note = {arXiv:2608.24122},
}

@article{runge,
 author = {Runge, Bernhard},
 title = {Endomorphism rings of abelian surfaces and projective models of their moduli spaces},
 fjournal = {T{\^o}hoku Mathematical Journal. Second Series},
 journal = {T{\^o}hoku Math. J. (2)},
 issn = {0040-8735},
 volume = {51},
 number = {3},
 pages = {283--303},
 year = {1999},
 language = {English},
 doi = {10.2748/tmj/1178224764},
 zbMATH = {1389721},
 Zbl = {0972.14017}
}

@misc{kanigeneralized,
 author = {Ernst Kani},
 title = {Generalized {H}umbert Schemes and
Intersections of {H}umbert Surfaces},
 year = {2019},
 note = {https://mast.queensu.ca/~kani/papers/interHum11.pdf},
}

@incollection{rotgershimura,
 author = {Rotger, Victor},
 title = {Shimura curves embedded in {Igusa}'s threefold},
 booktitle = {Modular curves and Abelian varieties. Based on lectures of the conference, Bellaterra, Barcelona, July 15--18, 2002},
 isbn = {3-7643-6586-2},
 pages = {263--276},
 year = {2004},
 publisher = {Basel: Birkh{\"a}user},
 language = {English},
 zbMATH = {2164186},
 Zbl = {1071.11035}
}

@misc{guoyang,
 author = {Guo, Jia-Wei and Yang, Yifan},
 title = {Class number relations arising from intersections of {Shimura} curves and {Humbert} surfaces},
 year = {2019},
 howpublished = {Preprint, {arXiv}:1903.07225 [math.{NT}]},
 url = {https://arxiv.org/abs/1903.07225},
 arXiv = {arXiv:1903.07225}
}

@article{linyang,
 author = {Lin, Yi-Hsuan and Yang, Yifan},
 title = {Quaternionic loci in {Siegel}'s modular threefold},
 fjournal = {Mathematische Zeitschrift},
 journal = {Math. Z.},
 issn = {0025-5874},
 volume = {295},
 number = {1-2},
 pages = {775--819},
 year = {2020},
 language = {English},
 doi = {10.1007/s00209-019-02372-z},
 zbMATH = {7203139},
 Zbl = {1468.11137}
}

@Article{humbert,
 Author = {Humbert, Georges},
 Title = {Sur les fonctions ab{\'e}liennes singuli{\`e}res. {Premier} {M{\'e}moire}.},
 FJournal = {Journal de Math{\'e}matiques Pures et Appliqu{\'e}es. 5. S{\'e}rie},
 Journal = {Journ. de Math. (5)},
 Volume = {5},
 Pages = {233--350},
 Year = {1899},
 Language = {French},
 URL = {https://eudml.org/doc/234233},
 zbMATH = {2668426},
 JFM = {30.0408.04}
}

@misc{gijongomez,
  title        = {On the {CM} Exception to a Generalization of the {S}téphanois Theorem},
  author       = {Desirée {Gijón Gómez}},
  year         = {2025},
  eprint       = {2505.08570},
  archivePrefix= {arXiv},
  primaryClass = {math.NT},
  doi          = {10.48550/arXiv.2505.08570},
  note        = {https://arxiv.org/abs/2505.08570}
}

@Article{KNRR,
 Author = {Kirschmer, Markus and Narbonne, Fabien and Ritzenthaler, Christophe and Robert, Damien},
 Title = {Spanning the isogeny class of a power of an elliptic curve},
 FJournal = {Mathematics of Computation},
 Journal = {Math. Comput.},
 ISSN = {0025-5718},
 Volume = {91},
 Number = {333},
 Pages = {401--449},
 Year = {2022},
 Language = {English},
 DOI = {10.1090/mcom/3672},
 zbMATH = {7446421},
 Zbl = {1486.14046}
}

@InCollection{gelin,
 Author = {G{\'e}lin, Alexandre and Howe, Everett and Ritzenthaler, Christophe},
 Title = {Principally polarized squares of elliptic curves with field of moduli equal to {{\(\mathbb{Q}\)}}},
 BookTitle = {ANTS XIII. Proceedings of the thirteenth algorithmic number theory symposium, University of Wisconsin-Madison, WI, USA, July 16--20, 2018},
 ISBN = {978-1-935107-02-6; 978-1-935107-03-3},
 Pages = {257--274},
 Year = {2019},
 Publisher = {Berkeley, CA: Mathematical Sciences Publishers (MSP)},
 Language = {English},
 DOI = {10.2140/obs.2019.2.257},
 zbMATH = {7721129},
 Zbl = {1517.11063}
}

@Article{kanihumbert,
 Author = {Kani, Ernst},
 Title = {The moduli spaces of {Jacobians} isomorphic to a product of two elliptic curves},
 FJournal = {Collectanea Mathematica},
 Journal = {Collect. Math.},
 ISSN = {0010-0757},
 Volume = {67},
 Number = {1},
 Pages = {21--54},
 Year = {2016},
 Language = {English},
 DOI = {10.1007/s13348-015-0148-9},
 zbMATH = {6550621},
 Zbl = {1353.14040}
}

@misc{borevich,
 author = {Borevich, Zenon I. and Faddeev, Dmitrii K.},
 title = {Representations of orders with a cyclic index},
 year = {1965},
 language = {English},
 howpublished = {Proc. {Steklov} {Inst}. {Math}. 80 (1965), 56-72 (1968); translation from {Tr}. {Mat}. {Inst}. {Steklov} 80, 51-65},
 zbMATH = {3286148},
 Zbl = {0179.06101}
}

@book{coxlib,
 author = {Cox, David A.},
 title = {Primes of the form {{\(x^2+ny^2\)}}. {Fermat}, class field theory, and complex multiplication. {Third} edition with solutions. {With} contributions by {Roger} {Lipsett}},
 edition = {3rd},
 fseries = {AMS Chelsea Publishing},
 series = {AMS Chelsea Publ.},
 volume = {387},
 isbn = {978-1-4704-7028-9; 978-1-4704-7183-5},
 year = {2022},
 publisher = {Providence, RI: American Mathematical Society (AMS)},
 language = {English},
 doi = {10.1090/chel/387},
 zbMATH = {7635220},
 Zbl = {1504.11002}
}

@phdthesis{gruenewald,
  author       = {David Gruenewald},
  title        = {Explicit Algorithms for {H}umbert Surfaces},
  school       = {University of Sydney},
  year         = {2008},
  month        = {December}
}

@article{kani-abelian,
 author = {Kani, Ernst},
 title = {Elliptic curves on abelian surfaces},
 fjournal = {Manuscripta Mathematica},
 journal = {Manuscr. Math.},
 issn = {0025-2611},
 volume = {84},
 number = {2},
 pages = {199--223},
 year = {1994},
 language = {English},
 doi = {10.1007/BF02567454},
 url = {https://eudml.org/doc/155988},
 zbMATH = {701163},
 Zbl = {0821.14023}
}

@phdthesis{kirthesis,
  author       = {K{\i}r, Harun},
  title        = {Curves of Genus 2 and Quadratic Forms},
  school       = {Queen’s University},
  year         = {2024}
}

@article {liu,
  author =       {Liu, Qing},
  TITLE =        {Conducteur et discriminant minimal de courbes de genre
                  {$2$}},
  JOURNAL =      {Compositio Math.},
  FJOURNAL =     {Compositio Mathematica},
  VOLUME =       94,
  YEAR =         1994,
  NUMBER =       1,
  PAGES =        {51--79},
}

@Article{shimura61:_compl_abelian,
  author =       {Shimura, Goro and Taniyama, Yutaka},
  title =        {Complex multiplication of {A}belian varieties and its
                  applications to number theory},
  journal =      {Publ. Math. Soc. Japan},
  year =         1961,
  volume =       6
}

@book{geerhilbert,
 author = {van der Geer, Gerard},
 title = {Hilbert modular surfaces},
 fseries = {Ergebnisse der Mathematik und ihrer Grenzgebiete. 3. Folge},
 series = {Ergeb. Math. Grenzgeb., 3. Folge},
 issn = {0071-1136},
 volume = {16},
 isbn = {3-540-17601-2},
 year = {1988},
 publisher = {Berlin etc.: Springer-Verlag},
 language = {English},
 zbMATH = {192950},
 Zbl = {0634.14022}
}

\end{document}